\documentclass[preprint,10pt]{elsarticle}
\usepackage{amsmath}
\usepackage{amssymb}
\usepackage{amsthm}
\usepackage{indentfirst}
\usepackage[latin1]{inputenc}
\usepackage{geometry}
\usepackage{amsfonts}
\usepackage{enumitem}
\usepackage{cancel}
\usepackage[bookmarks=false,colorlinks=true,linkcolor=black,citecolor=black,filecolor=black,urlcolor=black]{hyperref}

\DeclareSymbolFont{largesymbols}{OMX}{yhex}{m}{n}
\DeclareMathAccent{\widehat}{\mathord}{largesymbols}{"62}

\newcommand{\MC}{G} % This is the macro for the metacyclic group \MC
\newcommand{\Q}{{\mathbb Q}}
\newcommand{\Z}{{\mathbb Z}}
\newcommand{\C}{{\mathbb C}}

\newcommand{\R}{{\mathbb R}}

\newcommand{\U}{\mathcal{U}}
\renewcommand{\O}{\mathcal{O}}
\newcommand{\Gal}{\textnormal{Gal}}
\newcommand{\vect}{\textnormal{vect}}
\newcommand{\Aut}{\textnormal{Aut}}
\newcommand{\End}{\textnormal{End}}
\newcommand{\Irr}{\textnormal{Irr}}
\newcommand{\GL}{\textnormal{GL}}
\newcommand{\SL}{\textnormal{SL}}
\newcommand{\tr}{\textnormal{tr}}
\newcommand{\Bass}[3]{u_{#1,#2}(#3)}
\newcommand{\suma}[1]{\widehat{#1}}
\newcommand{\Hc}{{\mathcal{H}}}
\newcommand{\inv}{{^{-1}}}
\newcommand{\GEN}[1]{\left\langle #1 \right\rangle}

\newtheorem{theorem}{Theorem}[section]

\newtheorem{lemma}[theorem]{Lemma}
\newtheorem{corollary}[theorem]{Corollary}
\newtheorem{proposition}[theorem]{Proposition}
\newtheorem{remark}[theorem]{Remark}

\newtheorem{example}[theorem]{Example}

\begin{document}

\begin{frontmatter}

\journal{Journal of Algebra}

\title{Group rings of finite strongly monomial groups: central units and primitive idempotents\tnoteref{tit}}

\tnotetext[tit]{The research is partially supported by the Research Foundation Flanders (FWO - Vlaanderen), Onderzoeksraad Vrije Universiteit Brussel, the grant PN-II-RU-TE-2009-1 project ID\_303 and by the Ministerio de Ciencia y Tecnolog\'{\i}a of Spain and Fundaci\'{o}n S\'{e}neca of Murcia.}

\author{Eric Jespers}
\address{Department of Mathematics, Vrije Universiteit Brussel,
Pleinlaan 2, 1050 Brussels, Belgium}
\ead{efjesper@vub.ac.be}

\author{Gabriela Olteanu}
\address{Department of Statistics-Forecasts-Mathematics, Babe\c s-Bolyai University,
Str. T. Mihali 58-60, 400591 Cluj-Napoca, Romania}
\ead{gabriela.olteanu@econ.ubbcluj.ro}

\author{\'{A}ngel del R\'io}
\address{Departamento de Matem\'{a}ticas, Universidad de Murcia,  30100 Murcia, Spain}
\ead{adelrio@um.es}

\author{Inneke Van Gelder \corref{cor}}
\address{Department of Mathematics, Vrije Universiteit Brussel,
Pleinlaan 2, 1050 Brussels, Belgium}
\ead{ivgelder@vub.ac.be}
\cortext[cor]{Corresponding author: Phone: +32 2 629 3551}


\date{\today}

\begin{abstract}
We compute the rank of the group of central units in the integral group ring $\Z G$ of a finite strongly monomial group $G$. The formula obtained is in terms of the strong Shoda pairs of $G$.
Next we construct a virtual basis of the group of central units of $\Z G$ for a class of groups $G$ properly contained in the finite strongly monomial groups. Furthermore, for another class of groups $G$ inside the finite strongly monomial groups, we give an explicit construction of a complete set of orthogonal primitive idempotents of $\Q G$.

Finally, we apply these results to describe finitely many generators of a subgroup of finite index in the group of units of $\Z G$, this for metacyclic groups $G$ of the form $G=C_{q^m}\rtimes C_{p^n}$ with $p$ and $q$ different primes and the cyclic group $C_{p^n}$ of order $p^n$ acting faithfully on the cyclic group $C_{q^m}$ of order $q^m$.
\end{abstract}

\end{frontmatter}

\section{Introduction}

For a finite group $G$ we denote by $\U(\Z G)$ the unit group of the integral group ring $\Z G$. Its group of central units is denoted by $\mathcal{Z}(\U(\Z G))$. It is well-known that $\mathcal{Z}(\U(\Z G))=\pm \mathcal{Z}(G) \times T$, where $T$ is a finitely generated free abelian group (see for example \cite[Corollary 7.3.3]{2002SehgalMilies}). In \cite{2005RitterSehgal} it is proved that the rank of the subgroup $T$ of $\mathcal{Z}(\U(\Z G))$ is the difference between the number of simple components of $\R G$ and the number of simple components of $\Q G$. For a finite strongly monomial group $G$, we
will give a description of this rank in terms of the strong Shoda pairs of $G$ (Theorem~\ref{rank}). Examples of strongly monomial groups are abelian-by-supersolvable groups \cite{Olivieri2004}. All monomial groups of order smaller than 1000 are strongly monomial and the smallest monomial non-strongly-monomial group is a group of order 1000, the 86-th one in the library of the GAP system \cite{2007Olteanu,GAP}.

Bass proved that if $C$ is a finite cyclic group, then the so-called Bass (cyclic) units generate a subgroup of finite index in $\U(\Z C)$ \cite{bass1966}. Using the Bass Independence Lemma, he also described a virtual basis for $\U(\Z C)$, i.e. an independent set of generators for a subgroup of finite index in $\U(\Z C)$. In these investigations the cyclotomic units show up and therefore the Bass units are a natural choice since they project to powers of cyclotomic units in each simple component of $\Q C$. Next Bass and Milnor proved this result for finite abelian groups. Their proof makes use of K-theory in order to reduce to group rings of cyclic groups. However, for arbitrary finite abelian groups, they did not describe an independent set of generators. Only recently such a virtual basis was described in \cite{2012JdRVG}. Its constructive proof is based on a delicate induction argument and hence avoids the use of K-theory and the Bass Independence Lemma.

In \cite{2012JdROVG} we have proved that the group generated by the so-called generalized Bass units contains a subgroup of finite index in $\mathcal{Z}(\U(\Z G))$ for any arbitrary finite strongly monomial group $G$. Note that no multiplicatively independent set for such a subgroup was obtained. However, we obtained an explicit description of a virtual basis of $\mathcal{Z}(\U(\Z G))$ when $G$ is a finite abelian-by-supersolvable group (and thus a strongly monomial group) such that every cyclic subgroup of order not a divisor of 4 or 6 is subnormal in $G$. Note that the latter does not apply to all finite split metacyclic groups $C_m\rtimes C_n$, for example if $n$ is a prime number and $C_m\rtimes C_n$ is not abelian then $C_n$ is not subnormal in $C_m\rtimes C_n$. On the other hand, Ferraz and Sim\'on did construct in \cite{FerrazSimon2008} a virtual basis of $\mathcal{Z}(\U(\Z (C_q\rtimes C_p)))$ for $p$ and $q$ odd and different primes. In our second main result (Theorem~\ref{basis}) we extend these results on the construction of a virtual basis of $\mathcal{Z}(\U(\Z G))$ to a class of finite strongly monomial groups containing the metacyclic groups $G=C_{q^m}\rtimes C_{p^n}$ with $p$ and $q$ different primes and $C_{p^n}$ acting faithfully on $C_{q^m}$. Our proof makes again use of strong Shoda pairs and the description of the Wedderburn decomposition of $\Q G$ obtained by Olivieri, del R\'io and Sim\'on in \cite{Olivieri2004}. Our approach is thus different from the one used in \cite{FerrazSimon2008}.

In \cite{Jespers2010} a complete set of matrix units (and in particular, of orthogonal primitive idempotents) of each simple component in the rational group algebra $\Q G$ is described for finite nilpotent groups $G$. The same is done in \cite{2011vangelder} for semisimple group algebras $F G$ of finite nilpotent groups $G$ over finite fields $F$. As an application one obtains a factorization of a subgroup of finite index of $\U(\Z G)$ into a product of three nilpotent groups, and one explicitly constructs finitely many generators for each of these factors. Examples were given to show that the construction of the orthogonal idempotents can not be extended to, for example, arbitrary finite metacyclic groups. In this paper, we are able to describe a complete set of matrix units for a class of finite strongly monomial groups containing the finite metacyclic groups $C_{q^m}\rtimes C_{p^n}$ with $C_{p^n}$ acting faithfully on $C_{q^m}$ (Theorem~\ref{idempotents}). For the latter groups we obtain as an application of these results (and the earlier results on central units) again an explicit construction of finitely many generators of three nilpotent subgroups that together generate a subgroup of finite index in $\U(\Z G)$ (Theorem~\ref{units pq}).

\section{Preliminaries}

Throughout, $G$ will be a finite group. Let $g$ be an element of $G$ of order $n$ and $k$ and $m$ positive integers such that $k^m\equiv 1 \mod n$. Then
$$\Bass{k}{m}{g}=(1+g+\dots + g^{k-1})^{m}+\frac{1-k^m}{n}(1+g+\dots+g^{n-1})$$ is a unit of the integral group ring $\Z G$; called a Bass unit. The units of this form were introduced in \cite{bass1966} and satisfy the following equalities (\cite[Lemma 3.1]{2006GoncalvesPassman}):
\begin{eqnarray}
 \label{Basseq1} \Bass{k}{m}{g}&=&\Bass{k_1}{m}{g}, \mbox{ if } k\equiv k_1 \mod n, \\
 \label{Basseq2} \Bass{k}{m}{g}\Bass{k}{m_1}{g}&=&\Bass{k}{m+m_1}{g},\\
 \label{Basseq3} \Bass{k}{m}{g}\Bass{k_1}{m}{g^k}&=&\Bass{kk_1}{m}{g}\text{ and } \\
 \label{Basseq4} \Bass{1}{m}{g}&=&1
\end{eqnarray}
for $g\in G$, $n=|g|$ and $k^m\equiv k_1^m \equiv k^{m_1} \equiv 1 \mod n$.
Moreover,
\begin{equation}\label{Basseq5}
\Bass{n-1}{m}{g} = (-g)^{-m}.
\end{equation}
By (\ref{Basseq2}) we have, for $r$ a positive integer,
\begin{equation}\label{Basseq6}
\Bass{k}{m}{g}^{r} = \Bass{k}{mr}{g}
\end{equation}
and from
(\ref{Basseq1}), (\ref{Basseq3}) and (\ref{Basseq5}) we deduce
    \begin{equation}\label{Basseq7}
    \Bass{n-k}{m}{g} = \Bass{k(n-1)}{m}{g} = \Bass{k}{m}{g} \Bass{n-1}{m}{g^k} = \Bass{k}{m}{g}g^{-km}
    \end{equation}
provided $(-1)^m\equiv 1 \mod n$.

Let $N$ be a normal subgroup of $G$. Using equations (\ref{Basseq1}) and (\ref{Basseq6}) together with the Chinese Remainder Theorem, it is easy to verify that some power of a given Bass unit in $\Z(G/N)$ is the natural image of a Bass unit in $\Z G$.

If $R$ is a unital associative ring and $G$ is a group then $R*^{\alpha}_{\tau} G$ denotes a crossed product with action $\alpha:G\rightarrow \Aut(R)$ and twisting (a two-cocycle) $\tau:G\times G \rightarrow \U(R)$ (see for example \cite{Passman1989}), i.e. $R*^{\alpha}_{\tau} G$ is the associative ring $\bigoplus_{g\in G} R u_g$ with multiplication given by the following rules: $u_g a = \alpha_g(a) u_g$ and $u_g u_h = \tau(g,h) u_{gh}$, for $a\in R$ and $g,h\in G$. In case $G$ is cyclic, say generated by $g$ of order $n$, then the crossed product $R*^{\alpha}_{\tau} G$ is completely determined by $\sigma=\alpha_g$ and $a=u_g^n$. In this case, as in \cite{Reiner1975}, the crossed product is simply denoted $(R,\sigma,a)$. Recall that a classical crossed product is a crossed product $L*^{\alpha}_{\tau} G$, where $L/F$ is a finite Galois extension, $G = \Gal(L/F)$ is the Galois group of the field extension $L/F$ and $\alpha$ is the natural action of $G$ on $L$. A classical crossed product $L *^{\alpha}_{\tau} G$ is denoted by $(L/F,\tau)$ \cite{Reiner1975}. If the twisting $\tau$ is cohomologically trivial, then the classical crossed product is isomorphic to a matrix algebra over its center. Moreover, when $\tau=1$ we get an explicit isomorphism. More precisely, denoting the matrix associated to an endomorphism $f$ in a basis $B$ as $[f]_B$, we have

\begin{theorem}\label{reiner}\cite[Corollary 29.8]{Reiner1975}
Let $L/F$ be a finite Galois extension and $n=[L:F]$. The classical crossed product $(L/F,1)$ is isomorphic (as $F$-algebra) to $M_n(F)$. Moreover, an isomorphism is given by
$$\begin{array}{rcccc}
\psi:(L/F,1) & \rightarrow & \End_F(L) & \rightarrow  & M_n(F)   \\
xu_{\sigma} & \mapsto & x'\circ \sigma & \mapsto & [x'\circ \sigma]_B,
\end{array}$$
for $x\in L$, $\sigma\in \Gal(L/F)$, $B$ an $F$-basis of $L$ and where $x'$ denotes multiplication by $x$ on $L$.
\end{theorem}

Our approach is making use of the description of the Wedderburn decomposition of the rational group algebra $\Q G$. We shortly recall the character-free method of Olivieri, del R\'io and Sim\'on \cite{Olivieri2004} to describe these simple components.

If $H$ is a subgroup of $G$ then $N_G(H)$ denotes the normalizer of $H$ in $G$. We use the exponential notation for conjugation: $a^b = b\inv a b$. For each $\alpha \in \Q G$, $C_G(\alpha)$ denotes the centralizer of $\alpha$ in $G$.

For a subgroup $H$ of $G$, let $\suma{H}=\frac{1}{|H|}\sum_{h\in H} h$. Clearly, $\suma{H}$ is an idempotent of $\Q G$ which is central if and only if $H$ is normal in $G$. If $K\lhd H\leq G$  and $K\neq H$ then let $$\varepsilon(H,K)=\prod (\suma{K}-\suma{M})=\suma{K}\prod (1-\suma{M}),$$ where $M$ runs through the set of all minimal normal subgroups of $H$ containing $K$ properly. We extend this notation by setting $\varepsilon(H,H)=\suma{H}$. Clearly $\varepsilon(H,K)$ is an idempotent of the group algebra $\Q G$. Let $e(G,H,K)$ be the sum of the distinct $G$-conjugates of $\varepsilon(H,K)$, that is, if $T$ is a right transversal of $C_G(\varepsilon(H,K))$ in $G$, then $$e(G,H,K)=\sum_{t\in T}\varepsilon(H,K)^t.$$ Clearly, $e(G,H,K)$ is a central element of $\Q G$ and if the $G$-conjugates of $\varepsilon(H,K)$ are orthogonal, then $e(G,H,K)$ is a central idempotent of $\Q G$.

A strong Shoda pair of $G$ is a pair $(H,K)$ of subgroups of $G$ with the properties that $K\leq H\unlhd N_G(K)$, $H/K$ is cyclic and a maximal abelian subgroup of $N_G(K)/K$ and the different $G$-conjugates of $\varepsilon(H,K)$ are orthogonal. In this case $C_G(\varepsilon(H,K))=N_G(K)$.

Let $\chi$ be an irreducible (complex) character of $G$. One says that $\chi$ is strongly monomial if there is a strong Shoda pair $(H,K)$ of $G$ and a linear character $\theta$ of $H$ with kernel $K$ such that $\chi=\theta^G$, the induced character of $G$. The group $G$ is strongly monomial if every irreducible character of $G$ is strongly monomial.

For finite strongly monomial groups, including finite abelian-by-supersolvable groups, all primitive central idempotents are of the form $e(G,H,K)$ with $(H,K)$ a strong Shoda pair of $G$. Note that different strong Shoda pairs can determine the same primitive central idempotent. Indeed, let $(H_1,K_1)$ and $(H_2,K_2)$ be two strong Shoda pairs of a finite group $G$. Then $e(G,H_1,K_1)=e(G,H_2,K_2)$ if and only if there is a $g\in G$ such that $H_1^g\cap K_2=K_1^g\cap H_2$ \cite{Olivieri2006}. In that case we say that $(H_1,K_1)$ and $(H_2,K_2)$ are equivalent as strong Shoda pairs of $G$. In particular, to calculate the primitive central idempotents of $G$ it is enough to consider only one strong Shoda pair in each equivalence class. We express this by saying that we take a complete and non-redundant set of strong Shoda pairs.

In \cite{Olivieri2004} more information was obtained on the strong Shoda pairs needed to describe the primitive central idempotents of the rational group algebra of a finite metabelian group. We recall the statement.

\begin{theorem}\label{SSPmetabelian}
Let $G$ be a metabelian finite group and let $A$ be a maximal abelian subgroup of $G$ containing the commutator subgroup $G'$. The primitive central idempotents of $\Q G$ are the elements of the form $e(G,H,K)$, where $(H,K)$ is a pair of subgroups of $G$ satisfying the following conditions:
\begin{enumerate}
\item \label{metabelian1}$H$ is a maximal element in the set $\{B\leq G \mid A\leq B \mbox{ and } B'\leq K\leq B\}$;
\item \label{metabelian2}$H/K$ is cyclic.
\end{enumerate}
\end{theorem}

For finite metacyclic groups, a more precise description of the primitive central idempotents of $\Q G$ in terms of the numerical information describing the group is given in \cite{Olivieri2006}.

The structure of the simple component $\Q Ge(G,H,K)$, with $(H,K)$ a strong Shoda pair of $G$, is given in the following proposition.

\begin{proposition}\label{SSP}\cite[Proposition 3.4]{Olivieri2004}
Let $(H,K)$ be a strong Shoda pair of a finite group $G$ and let $k=[H:K]$, $N=N_G(K)$, $n=[G:N]$, $yK$ a generator of $H/K$ and $\phi:N/H\rightarrow N/K$ a left
inverse of the canonical projection $N/K\rightarrow N/H$. Then $\Q Ge(G,H,K)$ is isomorphic to $M_n(\Q(\zeta_k)*_{\tau}^{\alpha} N/H)$ and the
action and twisting are given by
\begin{eqnarray*}
\alpha_{nH}(\zeta_k) &=& \zeta_k^i, \mbox{ if } yK^{\phi(nH)}=y^iK \mbox{ and}\\
\tau(nH,n'H) &=& \zeta_k^j, \mbox{ if }  \phi(nn'H)\inv\phi(nH)\phi(n'H)=y^jK,
\end{eqnarray*}
for $nH,n'H\in N/H$, integers $i$ and $j$ and $\zeta_k$ a complex primitive $k$-th root of unity.
\end{proposition}

Let $n$ be a positive integer, $C_n=\GEN{g}$, a cyclic group of order $n$ and $\zeta_n$ a complex primitive $n$-th root of unity. Then there are isomorphisms
    $$\begin{array}{ccccc} \Gal(\Q(\zeta_n)/\Q) & \rightarrow & \U(\Z/n\Z) & \rightarrow & \Aut(C_n) \\
        (\zeta_n\mapsto \zeta_n^r) & \mapsto & r & \mapsto & (\phi_r:g\rightarrow g^r)
    \end{array} $$
Throughout the paper we will abuse the notation and consider these isomorphisms as equalities so that, e.g. a subgroup $H$ of a cyclic group of order $\varphi(n)$ will be identified with a subgroup of $\U(\Z/n\Z)$ and with $\Gal(\Q(\zeta_n)/\Q(\zeta_n)^H)$. In particular, with the notation of Proposition~\ref{SSP}, the action $\alpha$ of the crossed product $\Q(\zeta_k)*_{\tau}^{\alpha} N/H$ in Proposition~\ref{SSP} is faithful. Therefore the crossed product $\Q(\zeta_k)*_{\tau}^{\alpha} N/H$ can be described as a classical crossed product $(\Q(\zeta_k)/F,\tau)$, where $F$ is the center of the
algebra and $\alpha$ is determined by the action of $N/H$ on $H/K$. In this way the Galois group $\Gal(\Q(\zeta_k)/F)$ can be identified with $N/H$ and with this identification $F=\Q(\zeta_k)^{N/H}$.

Theorem~\ref{SSPmetabelian} and Proposition~\ref{SSP} allow one to easily compute the primitive central idempotents and the Wedderburn components of the rational group algebra of a finite metacyclic group. Recall this is a group $S$ having a normal cyclic subgroup $N=\GEN{a}$ such that $S/N=\GEN{bN}$ is cyclic. Every finite metacyclic group $S$ has a presentation of the form $$S=\GEN{a,b\mid a^m=1,\ b^n=a^t,\ a^b=a^r},$$ where $m,n,t,r$ are integers satisfying the conditions $r^n \equiv 1 \mod m$ and $m\mid t(r-1)$. Define $\sigma\in\Aut(\GEN{a})$ as $\sigma(a)=a^b=a^r$. Let $u$ be the order of $\sigma$. Then, $u\mid n$.

We finish this section with recalling some well-known results on orders. A subring $\O$ of a finite dimensional semisimple $\Q$-algebra $A$ is called an order if it is a finitely generated $\Z$-module such
that $\Q\O=A$. For example if $G$ is a finite group then $\Z G$ is an order in $\Q G$ and $\mathcal{Z}(\Z G)$ is an order in $\mathcal{Z}(\Q G)$. The intersection of two orders in $A$ is again an order in $A$ and if $\O_1\subseteq \O_2$ are orders in $A$ then the index of their unit groups $[\U(\O_2):\U(\O_1)]$ is finite (see \cite[Lemma 4.2, Lemma 4.6]{Sehgal1993}). Moreover, the unit group of an order in $A$ is finitely generated \cite{1962BorelHarishChandra}. Finally recall that in a finitely generated abelian group replacing generators by powers of themselves yields a subgroup of finite index. We will use these properties several times in the proofs without explicit reference.

\section{The group of central units of $\Z G$ for a finite strongly monomial group $G$}\label{virtual}

In this section we will focus on strongly monomial groups $G$ such that there is a complete and non-redundant set of strong Shoda pairs $(H,K)$ of $G$ with the property that $[H:K]$ is a prime power. For this class of strongly monomial groups, we construct a virtual basis for the group $\mathcal{Z}(\U(\Z G))$. A virtual basis of an abelian group $A$ is a set of multiplicatively independent elements of $A$ which generate a subgroup of finite index in $A$. It will be clear that the construction of units in the basis is inspired by the construction of Bass units.

As it will be shown in Section~\ref{SectionMetacyclic}, the conditions in the statement of Theorem~\ref{basis} on the strong Shoda pairs of the group $G$ are fulfilled when $G$ is a metacyclic group of the type $C_{q^m}\rtimes C_{p^n}$ with $C_{p^n}$ acting faithfully on $C_{q^m}$. However, the class of strongly monomial groups such that there is a complete and non-redundant set of strong Shoda pairs $(H,K)$ of $G$ with the property that $[H:K]$ is a prime power, is a wider class. For example the alternating group $A_4$ of degree 4 satisfies the condition and it is not metacyclic (and not nilpotent). Although, not all strongly monomial groups have only strong Shoda pairs with prime power index. It can be shown that all strong Shoda pairs of the dihedral group $D_{2n}=\GEN{a,b\mid a^n=b^2=1,\ a^b=a^{-1}}$ (respectively, the quaternion group $Q_{4n}=\GEN{x,y \mid x^{2n} = y^4 = 1,\ x^n = y^2,\ x^y = x^{-1}}$) have prime power index if and only if $n$ is a power of a prime (respectively, $n$ is a power of 2).

We begin by determining the number of elements of a virtual basis, i.e. the rank, of $\mathcal{Z}(\U(\Z G))$. Let $G$ be a finite group and let $g,h\in G$. Recall that $g$ and $h$ are said to be $\R$-conjugate (respectively, $\Q$-conjugate) if $g$ is conjugate to either $h$ or $h\inv$ (respectively, to $h^r$ for some integer $r$ coprime with the order of $h$). This defines two equivalence relations on $G$ and their equivalence classes are called real and rational conjugacy classes of $G$, respectively. Using a Theorem of Berman and Witt and Dirichlet's Unit Theorem one can prove that the rank of $\mathcal{Z}(\U(\Z G))$ for $G$ a finite group is the difference between the number of real conjugacy classes and rational conjugacy classes of $G$. Furthermore the number of real (respectively, rational) conjugacy classes of $G$ coincides with the number of simple components of the Wedderburn decomposition of $\R G$ (respectively, $\Q G$) (see \cite[Theorem~42.8]{1962CurtisReiner}, \cite{2005RitterSehgal} and \cite{Ferraz2004}).

Let $G$ be a finite strongly monomial group. Then, by Proposition~\ref{SSP}, one obtains the following description of the Wedderburn decomposition of $\Q G$: $$\bigoplus_{(H,K)}M_{[G:N]}(\Q(\zeta_{[H:K]})* N/H),$$ with $(H,K)$ running through a complete and non-redundant set of strong Shoda pairs of $G$ and $N=N_G(K)$ (we use notations as in Proposition~\ref{SSP}).

Using the properties of the group of units of orders one deduces that the rank of $\mathcal{Z}(\U(\Z G))$ is the sum of the ranks of the groups of central units in orders of the simple components, and these are the ranks of the groups of units of the fixed rings $\Z[\zeta_{[H:K]}]^{N/H}$.

Consider the center $F=\Q(\zeta_{[H:K]})^{N/H}$ of the simple component $M_{[G:N]}(\Q(\zeta_{[H:K]})* N/H)$. Clearly, $$[F:\Q]=\frac{[\Q(\zeta_{[H:K]}):\Q]}{[\Q(\zeta_{[H:K]}):F]}=\frac{\varphi([H:K])}{[N:H]}.$$ Since $F$ is a Galois extension of $\Q$, we know that $F$ is either totally real or totally complex. If $F$ is totally real, then $F$ is contained in the maximal real subfield $\Q(\zeta_{[H:K]}+\zeta_{[H:K]}\inv)$ of $\Q(\zeta_{[H:K]})$. This happens if and only if the Galois group $N/H$ contains complex conjugation, which means that $hh^n \in K$ for some $n\in N$ and $h$ such that $H=\GEN{h,K}$. Now using the Dirichlet Unit Theorem, we obtain at once an appropriate rank computation.

\begin{theorem}\label{rank}
 Let $G$ be a finite strongly monomial group. Then the rank of $\mathcal{Z}(\U(\Z G))$ equals  $$\sum_{(H,K)} \left(\frac{\varphi([H:K])}{k_{(H,K)}[N:H]}-1\right),$$ where $(H,K)$ runs through a complete and non-redundant set of strong Shoda pairs of $G$, $h$ is such that $H=\GEN{h,K}$ and $$k_{(H,K)}=\left\{\begin{array}{ll} 1, & \mbox{if } hh^n\in K \mbox{ for some } n\in N_G(K); \\ 2, & \mbox{otherwise}.\end{array} \right.$$
\end{theorem}

In order to describe a virtual basis for $\mathcal{Z}(\U(\Z G))$ we need to construct some units.

Let $R$ be an associative ring with identity. Let $x\in R$ be a torsion unit of order $n$. Let $C_n=\GEN{g}$, a cyclic group of order $n$. Then the map $g\mapsto x$ induces a ring homomorphism $\Z\GEN{g}\rightarrow R$. If $k$ and $m$ are positive integers with $k^m\equiv 1\mod n$, then the element $$\Bass{k}{m}{x}=(1+x+\dots + x^{k-1})^{m}+\frac{1-k^m}{n}(1+x+\dots+x^{n-1})$$ is a unit in $R$ since it is the image of a Bass unit in $\Z\GEN{g}$. In particular, if $G$ is a finite group, $M$ a normal subgroup of $G$, $g\in G$ and $k$ and $m$ positive integers such that $\gcd(k,|g|)=1$ and
$k^m \equiv 1 \mod |g|$. Then we have $$\Bass{k}{m}{1-\suma{M}+g\suma{M}} = 1-\suma{M}+\Bass{k}{m}{g}\suma{M}.$$ Observe that any element $b=\Bass{k}{m}{1-\suma{M}+g\suma{M}}$ is an invertible element of $\Z G(1-\suma{M})+\Z G \suma{M}$. As this is an order in $\Q G$, there is a positive integer $n$ such that $b^n\in \U(\Z G)$. Let $n_{G,M}$ denote the minimal positive integer satisfying this condition for all $g\in G$. Then we call the element $$\Bass{k}{m}{1-\suma{M}+g\suma{M}}^{n_{G,M}}=\Bass{k}{mn_{G,M}}{1-\suma{M}+g\suma{M}}$$ a generalized Bass unit based on $g$ and $M$ with parameters $k$ and $m$. Note that we obtain the classical Bass units of $\Z G$ when $M$ is the trivial group.

Bass units of $\Z\GEN{g}$ project to powers of cyclotomic units in the simple components of $\Q\GEN{g}$. Hence we need to state some facts on these units. Let $\zeta_n$ denote a complex root of unity of order $n$. If $n>1$ and $k$ is an integer coprime with $n$ then $$\eta_{k}(\zeta_n)=\frac{1-\zeta_n^k}{1-\zeta_n}=1+\zeta_n+\zeta_n^2+\dots+\zeta_n^{k-1}$$ is a unit of $\Z[\zeta_n]$. We extend this notation by setting $$\eta_k(1)=1.$$ The units of the form $\eta_k(\zeta_n^j)$, with $j,k$ and $n$ integers such that $\gcd(k,n)=1$, are called the cyclotomic units of $\Q(\zeta_n)$.

It is easy to verify that the cyclotomic units satisfy the following equalities:
\begin{eqnarray}
 \label{Cycleq1} \eta_k(\zeta_n)&=&\eta_{k_1}(\zeta_n), \mbox{ if } k\equiv k_1 \mod n, \\
 \label{Cycleq2} \eta_{kk_1}(\zeta_n) &=& \eta_k(\zeta_n)\eta_{k_1}(\zeta_n^k),\\
 \label{Cycleq3} \eta_1(\zeta_n)&=&1, \\
 \label{Cycleq5} \eta_{n-k}(\zeta_n) &=& -\zeta_n^{-k}\eta_k(\zeta_n),
\end{eqnarray}
for $n>1$ and both $k$ and $k_1$ coprime with $n$.

In order to compute a virtual basis of $\mathcal{Z}(\U(\Z G))$, we will ``cover'' the central integral units in the different simple components by using generalized Bass units. This will lead to a final description of the central units up to finite index. Indeed, take an arbitrary central unit $u$ in $\mathcal{Z}(\U(\Z G))$. Then we can write this element as follows $$u=\sum_{(H,K)}u e(G,H,K)=\prod_{(H,K)} (1-e(G,H,K)+u e(G,H,K)),$$ where $(H,K)$ runs through a complete and non-redundant set of strong Shoda pairs of $G$. Hence it is necessary and sufficient to construct a set of multiplicatively independent units in the center of each order $\Z Ge(G,H,K) + \Z(1-e(G,H,K))$.

The center of $\Z Ge(G,H,K) + \Z(1-e(G,H,K))$ and $\Z[\zeta_{[H:K]}]^{N_{G}(K)/H} + \Z(1-e(G,H,K))$ are both orders in the center of $\Q Ge(G,H,K) + \Q(1-e(G,H,K))$ and therefore their unit groups are
commensurable, i.e. their intersection has finite index in both of them. Hence, we are interested in the units of $\Z[\zeta_{[H:K]}]^{N_{G}(K)/H}$ and furthermore in the units of $\Z G$ projecting to units in $\Z[\zeta_{[H:K]}]^{N_{G}(K)/H}$ and trivially in the other components.

It is known that the set $\{\eta_k(\zeta_{p^n})\mid 1<k< \frac{p^n}{2},\ p\nmid k\}$ generates a free abelian subgroup of finite index in $\U(\Z[\zeta_{p^n}])$ when $p$ is prime \cite[Theorem 8.2]{1982Washington}. For a subgroup $A$ of $\Aut(\GEN{\zeta_{p^n}})\simeq\Gal(\Q(\zeta_{p^n})/\Q)$ and $u\in \Q(\zeta_{p^n})$, we define $\pi_A(u)$ to be $\prod_{\sigma\in A } \sigma(u)$.
Since, by assumption, $[H:K]$ equals a prime power, say $p^n$, it is well-known that $\Aut(H/K)$ is cyclic, unless $p=2$ and $n\ge 3$ in which case $\Aut(H/K)=\GEN{\phi_5}\times\GEN{\phi_{-1}}$. Since the Galois group $N_G(K)/H$ of $\Q(\zeta_{p^n})/\Q(\zeta_{p^n})^{N_G(K)/H}$ is a subgroup of $\Aut(H/K)$, it follows that $N_G(K)/H$ is either $\GEN{\phi_r}$ for some $r$ or $\GEN{\phi_r}\times\GEN{\phi_{-1}}$ for some $r\equiv 1 \mod 4$. We simply denote $\pi_{N_G(K)/H}$ by $\pi$ and have
\begin{eqnarray*}
\pi(u) = N_{\Q(\zeta_{p^n})/\Q(\zeta_{p^n})^{N_G(K)/H}}(u)
= \prod_{\sigma\in N_G(K)/H } \sigma(u)
= \prod_{i\in N_G(K)/H} u^i
\end{eqnarray*}
for $u\in \Q(\zeta_{p^n})$. Observe that we are making use of the abuse of notation to identify $N_G(K)/H$ with $\Gal\left(\Q(\zeta_{p^n})/\Q(\zeta_{p^n})^{N_G(K)/H}\right)$ and with a subgroup of $\U(\Z/[H:K]\Z)$. We will also need the following Lemma.

\begin{lemma}\label{norm}
Let $A$ be a subgroup of $\Aut(\GEN{\zeta_{p^n}})$. Let $I$ be a set of coset representatives of $\U(\Z/{p^n\Z})$ modulo $\GEN{A,\phi_{-1}}$ containing 1. Then the set $$\{\pi_{A}\left(\eta_k(\zeta_{p^n})\right)\mid k\in I\setminus\{1\}\}$$ is a virtual basis of $\U\left(\Z[\zeta_{p^n}]^{A}\right)$.
\end{lemma}
\begin{proof}
Assume $A=\GEN{\phi_r}$ or $A=\GEN{\phi_r}\times\GEN{\phi_{-1}}$. In both cases, we set $l=|\GEN{\phi_r}|$.

The arguments in the paragraph before Theorem~\ref{rank} show that the unit group $\U\left(\Z[\zeta_{p^n}]^{A}\right)$ has rank $\frac{p^{n-1}(p-1)}{ld}-1=|I|-1$, where $d=2$ if $-1\not\in\GEN{r}$ and $d=1$ otherwise. Moreover $ld=|\GEN{A,\phi_{-1}}|$.

We noticed before that the cyclotomic units of the form $\eta_k(\zeta_{p^n})$ generate a subgroup of finite index in $\U(\Z[\zeta_{p^n}])$. Therefore, for every unit $u$ of $\Z[\zeta_{p^n}]^{A}$, $u^m=\prod_{i=1}^h \eta_{k_i}(\zeta_{p^n})$ for some integers $m,k_1,\dots,k_h$. Then, $u^{m|A|}=\pi_{A}(u) = \prod_{i=1}^h \pi_{A}(\eta_{k_i}(\zeta_{p^n}))$. Hence, it is clear that $$\{\pi_{A}(\eta_k(\zeta_{p^n}))\mid k\in\U(\Z/{p^n\Z})\}$$ generates a subgroup of finite index in $\U\left(\Z[\zeta_{p^n}]^{A}\right)$.

First consider the case when $d=1$ (i.e. $-1\in\GEN{r}$). Because of (\ref{Cycleq2}), we have $$\eta_{r^ti}(\zeta_{p^n})=\eta_{i}(\zeta_{p^n})\eta_{r^t}(\zeta_{p^n}^i), $$
for $i\in I$ and $0\leq t\leq l-1$. Note that $\pi_{A}(\eta_{r^t}(\zeta_{p^n}^i))=\pi_{A}(\eta_{r^t}(\zeta_{p^n}^{ir^t}))$, for $0\leq t\leq l-1$. Again by (\ref{Cycleq2}), we deduce that $$(\pi_{A}(\eta_{r^t}(\zeta_{p^n}^i)))^2 = \pi_{A}(\eta_{r^t}(\zeta_{p^n}^i))\pi_{A}(\eta_{r^t}(\zeta_{p^n}^{ir^t}))=\pi_{A}(\eta_{r^{2t}}(\zeta_{p^n}^i)),$$ and hence it follows by induction and equations (\ref{Cycleq1}) and (\ref{Cycleq3}) that $\pi_{A}(\eta_{r^t}(\zeta_{p^n}^i))$ is torsion. Hence $$\{\pi_{A}(\eta_{i}(\zeta_{p^n}))\mid i\in I\}$$ still generates a subgroup of finite index in $\U\left(\Z[\zeta_{p^n}]^{A}\right)$.

Now consider the case when $d=2$ (i.e. $-1\not\in\GEN{r}$). Let $J=I\cup -I$. Then $J$ is a set of coset representatives of $\U(\Z/{p^n\Z})$ modulo $\GEN{r}$. By the same arguments as above, we can deduce that $$\{\pi_{A}(\eta_{k}(\zeta_{p^n}))\mid k\in J\}$$ generates a subgroup of finite index in $\U\left(\Z[\zeta_{p^n}]^{A}\right)$. If $i\in I$ then, by equation (\ref{Cycleq5}), we have that $\pi_{A}(\eta_{-i}(\zeta_{p^n}))=\pi_{A}(-\zeta_{p^n}^{-i}) \pi_{A}(\eta_i(\zeta_{p^n}))$ and $\pi_{A}(-\zeta_{p^n}^{-i})$ is of finite order. Thus $$\{\pi_{A}(\eta_i(\zeta_{p^n}))\}\mid i\in I\}$$ still generates a subgroup of finite index in $\U\left(\Z[\zeta_{p^n}]^{A}\right)$.

Now, in both cases, by equation (\ref{Cycleq3}), we can exclude $k=1$. And since the size now coincides with the rank, the set $$\{\pi_{A}(\eta_{k}(\zeta_{p^n}))\mid k\in I\setminus\{1\}\}$$ is a virtual basis of $\U\left(\Z[\zeta_{p^n}]^{A}\right)$.
\end{proof}

The next Lemma and Proposition give a link with generalized Bass units. It is a translation of Proposition 4.2 in \cite{2012JdRVG} to the language of generalized Bass units. For the formulation we need some notation.

Let $H$ be a finite group, $K$ a subgroup of $H$ and $g\in H$ such that $H/K=\GEN{gK}$. Put $\Hc=\{L\leq H \mid K\leq L\}$. For every $L\in \Hc$, we fix a linear representation $\rho_L$ of $H$ with kernel $L$. Note that $\rho_L(\Q H)=\Q\left( \zeta_{[H:L]} \right)$ and $\rho_L(g)=\zeta_{[H:L]}$.

The following Lemma is a direct consequence of \cite[Lemma 2.1]{2012JdRVG} and the natural isomorphism $\Q (H/K)\simeq \Q H\suma{K}$.

\begin{lemma}
Let $K$ and $g\in H$ be as above, $L\in \Hc$ and $M$ an arbitrary subgroup of $H$. Let $l=|L\cap M|$, $t=[M:L\cap M]$ and let $k$ and $m$ be positive integers such that $(k,t)=1$ and $k^m \equiv 1 \mod |gu|$ for every $u\in M$.
Then
    $$\prod_{u\in M} \rho_L(\Bass{k}{mn_{H,K}}{gu\suma{K} +1-\suma{K}}) = \eta_k(\rho_L(g)^t)^{lmn_{H,K}}.$$
\end{lemma}

Let $H$ be a finite group and $K$ a subgroup of $H$ such that $H/K=\GEN{gK}$ is a cyclic group of order $p^n$. It follows that the subgroups of $H/K$ form a chain, hence $\Hc=\{L \mid K\leq L\leq H\}=\{H_j=\GEN{g^{p^{n-j}},K}\mid 0\leq j\leq n\}$. A crucial property to prove the next result.

Let $k$ be a positive integer coprime with $p$ and let $r$ be an arbitrary integer. For every $0\le j \le s \le n$ we construct recursively the following products of generalized Bass units of $\Z H$:
$$c_s^s(H,K,k,r)=1,$$ and, for $0\le j\le s-1$,
\begin{eqnarray*}
 c_j^s(H,K,k,r)&=&\left(\prod_{h\in H_j}\Bass{k}{O_{p^n}(k)n_{H,K}}{g^{rp^{n-s}}h\suma{K} +1-\suma{K}}\right)^{p^{s-j-1}}  \\ &&\left( \prod_{l=j+1}^{s-1}c_l^s(H,K,k,r)\inv \right) \left(\prod_{l=0}^{j-1} c_l^{s+l-j}(H,K,k,r)\inv\right).
\end{eqnarray*}
Here, for each positive integer $l$, and each $k$ coprime to $l$, we denote by $O_l(k)$ the multiplicative order of $k$ modulo $l$. We also agree that by definition an empty product equals 1.

The proof of the following statement is identical to that of Proposition 4.2 in \cite{2012JdRVG}.

\begin{proposition}\label{projections}
Let $H$ be a finite group and $K$ a subgroup of $H$ such that $H/K=\GEN{gK}$ is cyclic of order $p^n$. Let $\Hc=\{L\leq H \mid K\leq L\}=\{H_j=\GEN{g^{p^{n-j}},K}\mid 0\leq j\leq n\}$.
Let $k$ be a positive integer coprime with $p$ and let $r$ be an arbitrary integer. Then \begin{equation}
\rho_{H_{j_1}}(c_j^s(H,K,k,r)) = \left\{
\begin{array}{ll}
\eta_k(\zeta^r_{p^{s-j}})^{O_{p^n}(k)p^{s-1}n_{H,K}}, & \text{if } j=j_1;\\
1, & \text{if } j\neq j_1.
\end{array}\right.
\end{equation}
for every $0\le j,j_1\le s\le n$.
\end{proposition}

We are now in a position to state our main Theorem on central units. Observe that we are again identifying $N_G(K)/H$ with a subgroup of $\U(\Z/[H:K]\Z)$ for a strong Shoda pair $(H,K)$ of $G$.

\begin{theorem}\label{basis}
Let $G$ be a strongly monomial group such that there is a complete and non-redundant set $\mathcal{S}$ of strong Shoda pairs $(H,K)$ of $G$ with the property that each $[H:K]$ is a prime power. For every $(H,K)\in \mathcal{S}$, let $T_K$ be a right transversal of $N_G(K)$ in $G$, let $I_{(H,K)}$ be a set of representatives of $\U(\Z/[H:K]\Z)$ modulo $\GEN{N_G(K)/H,-1}$ containing $1$ and let $[H:K]=p_{(H,K)}^{n_{(H,K)}}$, with $p_{(H,K)}$ prime.  Then $$\left\{ \prod_{t\in T_K}\ \prod_{x\in N_G(K)/H} c_0^{n_{(H,K)}}(H,K,k,x)^t : (H,K) \in \mathcal{S}, k\in I_{(H,K)}\setminus \{1\} \right\}$$ is a virtual basis of $\mathcal{Z}(\U(\Z G))$.
\end{theorem}
\begin{proof}
It is sufficient to prove for every $(H,K)\in \mathcal{S}$ that $$\left\{ \prod_{t\in T}\ \prod_{x\in N/H} c_0^{n}(H,K,k,x)^t : k\in I\setminus \{1\} \right\}$$ is a virtual basis of the center of $1-e(G,H,K)+\U(\Z G e(G,H,K))$, with $N=N_G(K)$, $T=T_K$, $I=I_{(H,K)}$ and $n=n_{(H,K)}$. To prove this, we may assume without loss of generality that $K$ is normal in $G$ (i.e. $N=G$). Indeed, assume we can compute a virtual basis $\{u_1,\ldots,u_s\}$ of the center of $1-\varepsilon(H,K)+\U(\Z N\varepsilon(H,K))$. Each $u_i$ is of the form $1-\varepsilon(H,K)+v_i\varepsilon(H,K)$ for some $v_i\in \Z N$ and $u_i\inv = 1-\varepsilon(H,K)+v_i'\varepsilon(H,K)$ for some $v_i'\in \Z N$. Then, $w_i=\prod_{t\in T} u_i^t=1-e(G,H,K)+\sum_{t\in T}v_i^t\varepsilon(H,K)^t$ is a unit in the center of $1-e(G,H,K)+\Z Ge(G,H,K)$ since the $\varepsilon(H,K)^t$ are mutually orthogonal idempotents and they also are orthogonal to $1-e(G,H,K)$. Then $w_1,\ldots,w_s$ are independent because so are $u_1,\ldots,u_s$ and they form a virtual basis of the center of $1-e(G,H,K)+\Z Ge(G,H,K)$.

From now on we thus will assume that $K$ is normal in $G$ and $[H:K]=p^n$ with $p$ prime. Thus $N=G$ and $T=\{1\}$. We have to prove that 	$$\left\{ \prod_{x\in G/H} c_0^{n}(H,K,k,x) : k\in I\setminus \{1\} \right\}$$ is a virtual basis of the center of $1-e(G,H,K)+\U(\Z G e(G,H,K))$.

Assume first that $H=G$. Then $\Q Ge(G,G,K)\simeq \Q( \zeta_{p^n})$. Consider the algebra $\Q Ge(G,G,K) + \Q(1-e(G,G,K))$ inside the algebra $\Q G\suma{K}+ \Q(1-\suma{K})$. By Proposition~\ref{projections}, the elements $c_{0}^n(G,K,k,1)$ project to $\eta_k(\zeta_{p^n})^{O_{p^n}(k)p^{n-1}n_{G,K}}$ in the simple component $\Q( \zeta_{p^n} )$ and trivially in all other components. Since we know that the set $\{\eta_k(\zeta_{p^n})\mid k\in I\setminus \{1\} \}$ is a virtual basis of $\U(\Z[\zeta_{p^n}])$ (Lemma~\ref{norm}), it follows that $\{c_{0}^n(G,K,k,1)\mid k\in I\setminus\{1\}\}$ is a virtual basis of $1-e(G,G,K)+\U(\Z Ge(G,G,K))$.

Now, assume that $G\ne H$ and consider the non-commutative simple component $\Q Ge(G,H,K) \simeq \Q H\varepsilon(H,K)*G/H$ with center $(\Q H\varepsilon(H,K))^{G/H}\simeq \Q(\zeta_{p^n})^{G/H}$. Consider the commutative algebra $(\Q H\varepsilon(H,K))^{G/H} + \Q(1-\varepsilon(H,K))$ inside the algebra $\Q H\suma{K}+ \Q(1-\suma{K})$. Since $H/K$ is a cyclic $p$-group, $G/H=\GEN{\phi_r}$ or $G/H=\GEN{\phi_r}\times\GEN{\phi_{-1}}$ for some $r$. Say $|\GEN{\phi_r}|=l$. By Lemma~\ref{norm}, the set $$\{\pi\left(\eta_k(\zeta_{p^n})\right)\mid k\in I\setminus\{1\}\}$$ is a virtual basis of $\U\left(\Z[\zeta_{p^n}]^{G/H}\right)$.

If $G/H$ is cyclic, by Proposition~\ref{projections} the elements $$c_0^n(H,K,k,1)c_0^n(H,K,k,r)\cdots c_0^n(H,K,k,r^{l-1})$$ project to $\pi(\eta_k(\zeta_{p^n}))^{O_{p^n}(k)p^{n-1}n_{H,K}}$ in the component $\Q( \zeta_{p^n} )^{G/H}$ and trivially in all other components of $\Q H$. Hence the set $$\{c_0^n(H,K,k,1)c_0^n(H,K,k,r)\cdots c_0^n(H,K,k,r^{l-1}) \mid k\in I\setminus\{1\}\}$$ is a virtual basis of $\mathcal{Z}(\U(\Z G e(G,H,K) +\Z(1-e(G,H,K))))$.

If $G/H$ is not cyclic, then the elements $$\prod_{i=0}^{l-1}\ \prod_{j=0}^1 c_0^n(H,K,k,r^i(-1)^j)$$ project to a power of $\pi(\eta_k(\zeta_{p^n}))$ in the component $\Q( \zeta_{p^n} )^{G/H}$ and trivially in all other components of $\Q H$. Hence also in this case we find a set $$\left\{\prod_{i=0}^{l-1}\ \prod_{j=0}^1 c_0^n(H,K,k,r^i(-1)^j) \mid k\in I\setminus\{1\}\right\}$$ which is a virtual basis of $\mathcal{Z}(\U(\Z G e(G,H,K) +\Z(1-e(G,H,K))))$ .

This finishes the proof, since we have now constructed a virtual basis in the center of each order $\Z Ge(G,H,K) + \Z(1-e(G,H,K))$, with $(H,K)\in \mathcal{S}$.
\end{proof}

\section{A complete set of orthogonal primitive idempotents in $\Q G$}

In this section we will focus on simple components of $\Q G$ of a finite group $G$ which are determined by a strong Shoda pair $(H,K)$ such that $\tau(nH,n'H)=1$ for all $n,n'\in N_G(K)$ (with notation as in Proposition~\ref{SSP}). For such a component, we describe a complete set of orthogonal primitive idempotents (and a complete set of matrix units). This construction is based on the isomorphism of Theorem~\ref{reiner} on classical crossed products with trivial twisting. Such a description, together with the description of the primitive central idempotent $e=e(G,H,K)$ determining the simple component, yields a complete set of irreducible modules. It also will allow us to construct units in Section~\ref{SectionMetacyclic} that determine a large subgroup of the full unit group $\U(\Z G)$.

Before we do so, we need a basis of $\Q(\zeta_{[H:K]})/\Q(\zeta_{[H:K]})^{N_G(K)/H}$ of the form $\{w^x\mid x\in N_G(K)/H\}$ with $w\in \Q(\zeta_{[H:K]})$. That such a basis exists follows from the well-known Normal Basis Theorem which states that if $E/F$ is a finite Galois extension, then there exists an element $w\in E$ such that $\{\sigma(w)\mid \sigma\in \Gal(E/F)\}$ is an $F$-basis of $E$, a so-called normal basis, whence $w$ is called normal in $E/F$.

\begin{theorem}\label{idempotents}
Let $(H,K)$ be a strong Shoda pair of a finite group $G$ such that $\tau(nH,n'H)=1$ for all $n,n'\in N_G(K)$. Let $\varepsilon=\varepsilon(H,K)$ and $e=e(G,H,K)$. Let $F$ denote the fixed subfield of $\Q H \varepsilon$ under the natural action of $N_G(K)/H$ and $[N_G(K):H]=n$.
Let $w$ be a normal element of $\Q H \varepsilon/F$ and $B$ the normal basis determined by $w$.
Let $\psi$ be the isomorphism between $\Q N_G(K) \varepsilon$ and the matrix algebra $M_{n}(F)$ with respect to the basis $B$ as stated in Theorem~\ref{reiner}.
Let $P,A\in M_{n}(F)$ be the matrices $$P= \left( \begin{array}{rrrrrr}
1 & 1 & 1 & \cdots & 1 & 1\\
1 & -1 & 0 & \cdots & 0 & 0\\
1 & 0 & -1 & \cdots & 0 & 0\\
\vdots & \vdots & \vdots & \ddots & \vdots & \vdots\\
1 & 0 & 0 & \cdots & -1 & 0\\
1 & 0 & 0 & \cdots & 0 & -1\\
\end{array} \right)
\quad \text{and} \quad
A= \left( \begin{array}{ccccc}
0 & 0 & \cdots & 0 & 1\\
1 & 0 & \cdots & 0 & 0\\
0 & 1 & \cdots & 0 & 0\\
\vdots & \vdots & \ddots & \vdots & \vdots\\
0 & 0 & \cdots & 0 & 0\\
0 & 0 & \cdots & 1 & 0\\
\end{array} \right).$$
Then $$\{x\widehat{T_1}\varepsilon x^{-1} \mid x\in T_2\GEN{x_e}\}$$ is a complete set of orthogonal primitive idempotents of $\Q G e$ where $x_e=\psi^{-1}(PAP^{-1})$, $T_1$ is a transversal of $H$ in $N_G(K)$ and $T_2$ is a right transversal of $N_G(K)$ in $G$. By $\widehat{T_1}$ we denote the element $\frac{1}{|T_1|}\sum_{t\in T_1}{t}$ in $\Q G$.
\end{theorem}
\begin{proof}
Consider the simple component $$\Q Ge \simeq M_{[G:N]}(\Q N \varepsilon)\simeq M_{[G:N]}\left( (\Q H\varepsilon/F,1) \right)$$ of $\Q G$ with $N=N_G(K)$. Without loss of generality we may assume that $K$ is normal in $G$ and hence $N=G$. Indeed, if we obtain a complete set of orthogonal primitive idempotents of $\Q N\varepsilon$, then the conjugates by the transversal $T_2$ of $N$ in $G$ will give a complete set of orthogonal primitive idempotents of $\Q Ge$ since $e=\sum_{t\in T_2}\varepsilon^t$ and different $\varepsilon^t$'s are orthogonal.

From now on we assume that $N=G$ and $e=\varepsilon$. Then $B=\{w^{gH} : g\in T_1\}$. Since $G/H$ acts on $\Q He$ via the induced conjugation action on $H/K$ it easily is seen that the action of $G/H$ on $B$ is regular. Hence it is readily verified that for each $g\in T_1$, $\psi(ge)$ is a permutation matrix, and
$$\psi(\widehat{T_1}e)=\frac{1}{n}\left( \begin{array}{ccccc}
1 & 1 & \cdots & 1 & 1\\
1 & 1 & \cdots & 1 & 1\\
1 & 1 & \cdots & 1 & 1\\
\vdots & \vdots & \ddots & \vdots & \vdots\\
1 & 1 & \cdots & 1 & 1\\
1 & 1 & \cdots & 1 & 1\\
\end{array} \right).$$
Clearly $\psi(\widehat{T_1}e)$ has eigenvalues $1$ and $0$, with respective eigenspaces $V_1=\vect\{(1,1,\dots,1)\}$ and $V_0=\vect\{(1,-1,0,\dots,0),(1,0,-1,\dots,0),\dots,(1,0,0,\dots,-1)\}$, where $\vect(S)$ denotes the vector space generated by the set $S$. Hence $$\psi(\widehat{T_1}e)=PE_{11}P^{-1},$$ where we denote by $E_{ij}\in M_{n}(F)$ the matrices whose entries are all 0 except in the $(i,j)$-spot, where it is 1. One knows that $\{E_{11},E_{22},\dots,E_{nn}\}$ and hence also $$\{\psi(\widehat{T_1}e)=PE_{11}P^{-1},PE_{22}P^{-1},\dots,PE_{nn}P^{-1}\}$$ forms a complete set of orthogonal primitive idempotents of $M_{n}(F)$. Let $y=\psi(x_e)=PAP^{-1}$. As $$E_{22}=AE_{11}A^{-1}, \dots, E_{nn}=A^{n-1}E_{11}A^{-n+1}$$ we obtain that $$\{\psi(\widehat{T_1}e),y\psi(\widehat{T_1}e)y^{-1}, \dots, y^{n-1}\psi(\widehat{T_1}e)y^{-n+1}\}$$ forms a complete set of orthogonal primitive idempotents of $M_{n}(F)$. Hence, applying $\psi^{-1}$ gives us a complete set of orthogonal primitive idempotents of $\Q G e$.
\end{proof}

Next we will describe a complete set of matrix units in a simple component $\Q G e(G,H,K)$ for a strong Shoda pair $(H,K)$ of a finite group $G$.
\begin{corollary}\label{matrix units}
Let $(H,K)$ be a strong Shoda pair of a finite group $G$ such that $\tau(nH,n'H)=1$ for all $n,n'\in N$. We use the notation of Theorem~\ref{idempotents} and for every $x,x'\in T_2\GEN{x_e}$, let $$E_{xx'}=x\widehat{T_1}\varepsilon x'^{-1}.$$ Then $\{E_{xx'}\mid x,x'\in T_2\GEN{x_e}\}$ is a complete set of matrix units in $\Q G e$, i.e. $e=\sum_{x\in T_2\GEN{x_e}} E_{xx}$ and $E_{xy}E_{zw}=\delta_{yz}E_{xw}$, for every $x,y,z,w\in T_2\GEN{x_e}$.

Moreover $E_{xx}\Q G E_{xx}\simeq F$, where $F$ is the fixed subfield of $\Q H\varepsilon$ under the natural action of $N/H$.
\end{corollary}
\begin{proof}
This follows at once from Theorem~\ref{idempotents} and the fact that $\Q Ge\simeq M_{[G:H]}(F)$.
\end{proof}

In order to obtain an internal description within the group algebra $\Q G$, one would like to write the element $x_e=\psi\inv(PAP^{-1})$ of Theorem~\ref{idempotents} in terms of group elements of $\Q G$.  It might be a hard problem to obtain a generic formula. One of the reasons being that we first need to describe a normal basis of $\Q(\zeta_{[H:K]})/\Q(\zeta_{[H:K]})^{N/H}$. In general this is difficult to do. However, one can find some partial results in the literature. For example Hachenberger \cite{2000Hachenberger} studied normal bases for cyclotomic fields $\Q(\zeta_{q^m})$ with $q$ an odd prime. Once this obstacle is overcome one can determine $x_e$ as follows. Denote by $\Delta:\C N\rightarrow \C$ the trace map $\sum_{g\in N}a_gg \mapsto a_1$. It is easy to see and well-known that $\Delta(\alpha)=\frac{1}{|N|}\chi_{\mbox{\small reg}}(\alpha)=\frac{1}{|N|}\sum_{\chi\in\Irr(N)}\chi(1)\chi(\alpha)$, where we denote by $\chi_{\mbox{\small reg}}$ the regular character of $N$ and by $\Irr(N)$ the set of irreducible complex characters of $N$. It follows that $x_e = \sum_{g \in N}  \Delta(x_e g\inv) g = \frac{1}{|N|} \sum_{g\in N} \sum_{\chi \in \Irr(N)} \chi(1) \chi(x_e g\inv )g$.
Because $\psi$ can be seen as the representation induced to $N$ by a linear character $H$ with kernel $K$, $\psi$ is an irreducible complex representation of $N$.
As we know that $x_e$ belongs to a simple component of $\Q N$, namely the only one on which $\psi$ does not vanish, and the primitive central idempotent of $\Q N$ of such component is the sum of the primitive central idempotents of $\C N$ associated to the irreducible characters of the form $\sigma \circ T\circ \psi$, with $\sigma\in \Gal(F/\Q)$, we deduce that $\chi(x_e g\inv)$ vanishes in all the irreducible characters
different from $\sigma \circ T \circ \psi$ with $\sigma \in \Gal(F/\Q)$, where $T$ is the map associating a matrix with its trace.
Thus $x_e = \frac{1}{|N|} \sum_{g\in N} \sum_{\sigma\in \Gal(F/\Q)} (\sigma \circ T \circ \psi)(1) (\sigma \circ T \circ \psi)(x_e g\inv )g = \frac{1}{|H|} \sum_{g\in N}  \tr_{F/\Q} (T (PAP^{-1} \psi(g\inv ))) g$. In case $G$ is a finite nilpotent group then these problems can be overcome using the structure of the group; and so one can deal with arbitrary finite nilpotent groups, even in the case when $\tau$ is not trivial.

 We now give an example, of a finite metacyclic group, to show that sometimes these problems also can be overcome using only basic linear algebra.
\begin{example}
Let $\MC=C_7\rtimes C_3=\GEN{a,b\mid a^7=1=b^3,\ a^b=a^2}$. Consider the primitive central idempotent $e=e(\MC,\GEN{a},1)=\varepsilon(\GEN{a},1)$ and the non-commutative simple component $\Q \MC e=\Q\GEN{a}e*\GEN{b}$ with trivial twisting. Consider the algebra isomorphism $\psi:\Q\GEN{a}e*\GEN{b} \simeq M_3(\Q(ae+a^2e+a^4e))$ with respect to $B=\{ae,a^2e,a^4e\}$, a normal basis of $\Q (ae)$ over $\Q(ae+a^2e+a^4e)$. Now we have $A=\psi(be)$ and in order to describe $x_e=\psi^{-1}(PAP^{-1})$ in terms of elements of $\Q \MC$, it is sufficient to write $\psi^{-1}(P)$ in terms of group ring elements. Write $\psi^{-1}(P)=\alpha_0+\alpha_1b+\alpha_2b^2$ with $\alpha_i\in \Q\GEN{a}e$ and solve the system
$$\left\{ \begin{array}{lcl}
(\alpha_0'+\alpha_1'\circ b+\alpha_2 '\circ b^2)(ae) &=& (a+a^2+a^4)e\\
(\alpha_0'+\alpha_1'\circ b+\alpha_2 '\circ b^2)(a^2e) &=& (a-a^2)e\\
(\alpha_0'+\alpha_1'\circ b+\alpha_2 '\circ b^2)(a^4e) &=& (a-a^4)e.
 \end{array}\right.$$
This can be done by writing each $\alpha_i=(x_{i,0}+x_{i,1}a+x_{i,2}a^2+x_{i,3}a^3+x_{i,4}a^4+x_{i,5}a^5)e$ with $x_{i,j}\in \Q$ and using the equality $(1+a+a^2+a^3+a^4+a^5+a^6)e=0$. This leads to a system of 18 linear equations in 18 variables. It can be verified that
\begin{eqnarray*}
\psi^{-1}(P) = && \left( -\frac{4}{7}-\frac{1}{14}a-\frac{1}{2}a^2-\frac{5}{14}a^3-\frac{1}{7}a^4-\frac{5}{14}a^5 \right)e \\
&+& \left( \frac{2}{7}-\frac{11}{14}a-\frac{3}{14}a^2-\frac{1}{2}a^3-\frac{1}{7}a^4-\frac{9}{14}a^5 \right)be\\
&+&  \left( \frac{2}{7}-\frac{9}{14}a-\frac{11}{14}a^2-\frac{9}{14}a^3+\frac{2}{7}a^4-\frac{1}{2}a^5 \right)b^2e .
\end{eqnarray*}
\end{example}

However, it is crucial that the twistings appearing in the simple components are trivial in order to make use of Theorem~\ref{reiner}. The following example shows that our methods can not be extended to, for example, $C_q\rtimes C_{p^2}$ with non-faithful action.
\begin{example}
Consider the group $\MC=C_{19}\rtimes C_9=\langle a,b\mid a^{19}=b^9=1,\ a^b=a^7\rangle$ and the strong Shoda pair $(\GEN{a,b^3},1)$. Let $\varepsilon=\varepsilon(\GEN{ab^3},1)$. The elements $1,b,b^2$ are coset representatives for $\GEN{ab^3}=\GEN{a,b^3}$ in $\MC$. Since $b^2\GEN{a,b^3}b^2\GEN{a,b^3}=b\GEN{a,b^3}$ and $b^3=(ab^3)^{19}$, we get that $\tau_{b^2\GEN{a,b^3},b^2\GEN{a,b^3}}=\zeta_{57}^{19}\neq 1$. Hence the twisting is not trivial, although it is cohomologically trivial.
\end{example}

This method yields a detailed description of a complete set of orthogonal primitive idempotents of $\Q G$ when $G$ is a strongly monomial group such that there exists a complete and non-redundant set of strong Shoda pairs $(H,K)$ satisfying $\tau(nH,n'H)=1$ for all $n,n'\in N_G(K)$. As we will show in the next section on metacyclic groups, the groups of the form $C_{q^m}\rtimes C_{p^n}$ with $C_{p^n}$ acting faithfully on $C_{q^m}$ do satisfy this condition on the strong Shoda pairs. However not all groups satisfying this condition on the twistings are metacyclic, for example the symmetric group $S_4$ and the alternating group $A_4$ of degree 4 have a trivial twisting in all Wedderburn components of their rational group rings and are not metacyclic (and not nilpotent). Trivially all abelian groups are included and it is also easy to prove that for all dihedral groups $D_{2n}=\GEN{a,b\mid a^n=b^2=1,\ a^b=a^{-1}}$ there exists a complete and non-redundant set of strong Shoda pairs with trivial twisting since the group action involved has order 2 and hence is faithful. On the other hand for quaternion groups $Q_{4n}=\GEN{x,y \mid x^{2n} = y^4 = 1,\ x^n = y^2,\ x^y = x^{-1}}$, one can verify that the strong Shoda pair $(\GEN{x},1)$ yields a non-trivial twisting.

\section{Application: Generators of a subgroup of finite index in $\U(\Z C_{q^m}\rtimes C_{p^n})$}\label{SectionMetacyclic}

In this section, we first describe the Wedderburn decomposition of $\Q G$, for the metacyclic groups of the form $C_{q^m}\rtimes C_{p^n}$ with $C_{p^n}$ acting faithfully on $C_{q^m}$ and $p$ and $q$ different primes. Next we construct a virtual basis of the group $\mathcal{Z}(\U(\Z \MC))$. This generalizes results from \cite{FerrazSimon2008} where the case $m=n=1$ is handled. As an application we describe finitely many generators of three nilpotent subgroups of $\U(\Z (C_{q^m}\rtimes C_{p^n}))$ that together generate a subgroup of finite index. If $p=q$ then such a result was obtained in \cite{Jespers2010}, even for arbitrary finite nilpotent groups.

Throughout this section $p$ and $q$ are different primes, $m$ and $n$ are positive integers and $G=\GEN{a}\rtimes \GEN{b}$ with $|a|=q^m$, $|b|=p^n$ and $\GEN{b}$ acts faithfully on $\GEN{a}$ (i.e. $C_{\GEN{b}}(a)=1$). Let $\sigma$ be the automorphism of $\GEN{a}$ given by $\sigma(a)=a^b$ and assume that $\sigma(a)=a^r$ with $r\in \Z$. As the kernel of the restriction map $\Aut(\GEN{a})\rightarrow \Aut\left(\GEN{a^{q^{m-1}}}\right)$ has order $q^{m-1}$ it intersects $\GEN{\sigma}$ trivially and therefore the restriction of $\sigma$ to $\GEN{a^{q^{m-1}}}$ also has order $p^n$. This implies that $q\equiv 1 \mod p^n$ and thus $q$ is odd. Therefore, $\Aut\left(\GEN{a^{q^j}}\right)(=\Gal(\Q(\zeta_{q^j})/\Q)=\U(\Z/q^j\Z))$ is cyclic for every $j=0,1,\dots,m$ and $\GEN{\sigma}$ is the unique subgroup of $\Aut(\GEN{a})$ of order $p^n$. So, for every $i=1,\dots,m$, the image of $r$ in $\Z/q^i \Z$ generates the unique subgroup of $\U(\Z/q^i \Z)$ of order $p^n$. In particular $r^{p^n}\equiv 1 \mod q^m$ and $r^{p^j}\not\equiv 1 \mod q$ for every $j=0,\ldots,n-1$. Therefore, $r\not\equiv 1 \mod q$ and hence $\MC'=\GEN{a^{r-1}}=\GEN{a}$. Using the description of strong Shoda pairs of $\MC$ given in Theorem~\ref{SSPmetabelian}, we get a complete and non-redundant set of strong Shoda pairs of $G$ consisting of two types:
\begin{enumerate}[label=\rm(\roman{*}), ref=\roman{*}]
\item \label{SP1} $\left(\MC,L_i:=\GEN{a,b^{p^i}}\right), \; i=0,\dots,n$,\\
\item \label{SP2} $\left(\GEN{a},K_j:=\GEN{a^{q^j}}\right), \; j=1,\dots,m$.
\end{enumerate}
Using the description of the associated simple algebra given in Proposition~\ref{SSP}, we obtain the following description of the simple components of $\Q G$:
\begin{enumerate}[label=\rm(\Roman{*}), ref=\Roman{*}]
\item \label{SC1} $\Q \MC\varepsilon\left(\MC, L_i\right)\simeq \Q\left( \zeta_{p^i} \right), \; i=0,\dots,n$,\\
\item \label{SC2} $\Q \MC\varepsilon\left(\GEN{a},K_j\right) \simeq \Q\left(\zeta_{q^j}\right)*C_{p^n}, \; j=1,\dots,m$.
\end{enumerate}

It is easy to verify that the twisting of the crossed product in (\ref{SC2}) is trivial and hence, by Theorem~\ref{reiner}, this simple component is isomorphic to $$M_{p^n}\left({\left(\Q\GEN{a}\varepsilon\left(\GEN{a},K_j\right)\right)}^{\GEN{b}}\right)  \simeq M_{p^n}\left( F_j \right),$$ where $F_j=\Q\left(\zeta_{q^j}\right)^{C_{p^n}}$, the fixed field of $\Q\left(\zeta_{q^j}\right)$ by the action of $C_{p^n}$. Furthermore $F_j$ is the unique subfield of index $p^n$ in $\Q\left(\zeta_{q^j}\right)$.

We first compute the rank of $\mathcal{Z}(\U(\Z \MC))$ using the formula from Theorem~\ref{rank} and the description of the strong Shoda pairs (\ref{SP1}, \ref{SP2}). When $p$ is odd, an easy computation shows that the rank equals $$\sum_{i=1}^{n} \left( \frac{p^{i-1}(p-1)}{2}-1 \right) + \sum_{j=1}^m \left(\frac{q^{j-1}(q-1)}{2p^n}-1\right) =\frac{p^n-1}{2}+\frac{q^m-1}{2p^n}-n-m,$$ because $r$ has odd order modulo $q^m$.

When $p=2$, the rank equals $$\sum_{i=2}^{n} \left( 2^{i-2}-1 \right) + \sum_{j=1}^m \left(\frac{q^{j-1}(q-1)}{2^n}-1\right) = 2^{n-1}+\frac{q^m-1}{2^n}-n-m,$$ since $a^{b^{2^{n-1}}}=a\inv$ because $r$ has even order modulo $q^m$.

Now we use Theorem~\ref{basis} to obtain a virtual basis of $\mathcal{Z}(\U(\Z G))$:
\begin{eqnarray*}&&\left\{ c_0^{i}(G,L_i,k,1) \mid i=1,\dots,n,\ 1<k<\frac{p^i}{2},\ p\nmid k\right\}\\ &\bigcup&
\left\{ \prod_{x=0}^{p^n-1} c_0^{j}(\GEN{a},K_j,k,r^x) \mid j=1,\dots,m,\ k\in I_j\setminus \{1\} \right\}\end{eqnarray*}
where $I_j$ is a set of coset representatives of $\U(\Z/q^j\Z)$ modulo $\GEN{r,-1}$ containing $1$. We claim that the units $c_0^{j}(\GEN{a},K_j,k,r^x)$, which are products of generalized Bass units, can be replaced by $c_{m-j}^{m}(\GEN{a},1,k,r^x)$, which are products of Bass units. Indeed, these units project trivially on the commutative algebra $\Q G\suma{a}$. Moreover, by Proposition~\ref{projections}, they project to the unit $$\pi(\eta_k(\zeta_{q^j}))^{O_{q^m}(k)q^{m-1}}$$ in the simple component $\Q G \varepsilon(\GEN{a},K_j) \cong \Q ( \zeta_{q^j})$ and trivially in all other components of $\Q G (1-\widehat{a})$. By Lemma~\ref{norm} the set $$\{\pi\left(\eta_k(\zeta_{q^j})\right)\mid k\in I_j\setminus\{1\}\}$$ is a virtual basis of $\U\left(\Z[\zeta_{q^j}]^{C_{p^n}}\right)$. This proves the claim.

We summarize these results in the following theorem.

\begin{theorem}\label{basisMC}
Let $\MC=C_{q^m}\rtimes C_{p^n}$ be a finite metacyclic group with $C_{p^n}=\GEN{b}$ acting faithfully on $C_{q^m}=\GEN{a}$ and with $p$ and $q$ different primes. Let $r$ be such that $a^b=a^r$.
For each $j=1,\dots,m$, let $I_j$ be a set of coset representatives of $\U(\Z/{q^j\Z})$ modulo $\GEN{r,-1}$.

If $p=2$, then
\begin{eqnarray*}
U&=&\left\{c_{0}^i(\MC,\GEN{a,b^{2^i}},k,1)\mid 1<k< 2^{i-1},\ 2\nmid k,\ i=2,\ldots,n\right\}\\
&\bigcup& \left\{c_{m-j}^m(\GEN{a},1,k,1)c_{m-j}^m(\GEN{a},1,k,r)\cdots c_{m-j}^m(\GEN{a},1,k,r^{2^n-1})\mid k\in I_j\setminus\{1\},\ j=1,\dots,m\right\}
\end{eqnarray*}
is a virtual basis of $\mathcal{Z}(\U(\Z \MC))$, consisting of $2^{n-1}+\frac{q^m-1}{2^n}-n-m$ units.

If $p$ is odd, then
\begin{eqnarray*}
U&=&\left\{c_{0}^i(\MC,\GEN{a,b^{p^i}},k,1)\mid 1<k< \frac{p^i}{2},\ p\nmid k,\ i=1,\ldots,n\right\}\\
&\bigcup& \left\{c_{m-j}^{m}(\GEN{a},1,k,1)c_{m-j}^{m}(\GEN{a},1,k,r)\cdots c_{m-j}^{m}(\GEN{a},1,k,r^{p^n-1})\mid k\in I_j\setminus\{1\},\ j=1,\dots,m\right\}
\end{eqnarray*}
is a virtual basis of $\mathcal{Z}(\U(\Z \MC))$, consisting of $\frac{p^n-1}{2}+\frac{q^m-1}{2p^n}-n-m$ units.
\end{theorem}

As an example we will apply our result to the metacyclic group $C_q\rtimes C_p$ to deduce a result of \cite{FerrazSimon2008}.
\begin{example}
Let $G=C_q\rtimes C_p=\GEN{a,b \mid a^q=1=b^p,\ a^b=a^r}$ be a metacyclic group of order $pq$, for $p$ and $q$ different odd primes. Let $I$ be a set of coset representatives of $\U(\Z/q\Z)$ modulo $\GEN{r,-1}$. Then
\begin{eqnarray*}
&&\left\{ c_0^1(G,\GEN{a},k,1)  \mid 1<k< \frac{p}{2} \right\}\\ &&\bigcup \left\{ c_0^1(\GEN{a},1,k,1)c_0^1(\GEN{a},1,k,r)\cdots c_0^1(\GEN{a},1,k,r^{p-1}) \mid k\in
I\setminus\{1\}\right\} \\
&=& \left\{ \Bass{k}{O_{p}(k)n_{G,\GEN{a}}}{1-\suma{a}+b\suma{a}}^q  \mid 1<k< \frac{p}{2} \right\} \\ &&\bigcup \left\{ \Bass{k}{O_{q}(k)}{a}\Bass{k}{O_{q}(k)}{a^r}\cdots \Bass{k}{O_{q}(k)}{a^{r^{p-1}}}) \mid k\in
I\setminus\{1\}\right\}
\end{eqnarray*}
is a virtual basis of $\mathcal{Z}(\U(\Z G))$, consisting of $$\frac{p-1}{2} + \frac{(q-1)}{2p}-2$$ units.
\end{example}

We will now describe a complete set of orthogonal primitive idempotents in each simple component of $\Q \MC$.

The non-commutative simple components of $\Q \MC$ are $\Q G\varepsilon(\GEN{a},K_j)\cong M_{p^n}(F_j)$ for $j=1,\dots,m$. Fix a normal element $w_j$ of $\Q(\zeta_{q^j})/F_j$ and let $B_j$ be the normal basis determined by $w_j$. Let $\psi_j:\Q G \varepsilon(\GEN{a},K_j) \rightarrow M_{p^n}(F_j)$ be the isomorphism given by Theorem~\ref{reiner} with respect to $B_j$. Then $\psi_j(b\varepsilon(\GEN{a},K_j))$ is the permutation matrix $A$ of Theorem~\ref{idempotents} and $\GEN{b}$ is a transversal of $\GEN{a}$ in $G$. By Corollary~\ref{matrix units} we know that
$$\{ x_j^h \widehat{b} x_j^{-k} : 1\le h,k \le p^n\}$$ is a complete set of matrix units of $\Q G\varepsilon(\GEN{a},K_j)$, where $x_j=\psi_j\inv(P)b\varepsilon(\GEN{a},K_j)\psi_j\inv(P)\inv$.

As application of the description of the matrix units in each simple component $\Q \MC e$ and of the description of the central units in $\Z \MC$, we construct explicitly generators for three nilpotent subgroups of $\U(\Z G)$ that together generate a subgroup of finite index.

Let $\O$ be an order in a division algebra $D$ and denote by $\GL_{n}(\O)$ the group of invertible matrices in $M_{n}(\O)$ and by $\SL_{n}(\O)$ its subgroup consisting of matrices of reduced norm $1$. For an ideal $Q$ of $\O$ we denote by $E(Q)$ the subgroup of $\SL_{n}(\O)$ generated by all $Q$-elementary matrices, that is $E(Q)=\langle I+ qE_{ij} \mid  q\in Q,\ 1\leq i,j\leq n,\ i\neq j,\ E_{ij} \textrm{ a matrix unit} \rangle$. We summarize the following theorems \cite[Theorem 21.1, Corollary 21.4]{Bass1964}, \cite[Theorem 2.4, Lemma 2.6]{Vaserstein1973}, \cite[Theorem 24]{Liehl1981} and \cite[Theorem]{Venkataramana1994}.

\begin{theorem}[Bass-Vaser{\v{s}}te{\u\i}n-Liehl-Venkataramana] \label{elementary}
If $n\geq 3$ then $[\SL_{n}(\mathcal{O}):E(Q)]<\infty$. If $\U(\O)$ is infinite then $[\SL_{2}(\mathcal{O}):E(Q)]<\infty$.
\end{theorem}

In order to state the next theorem, it is convenient to introduce the notation of class sum. Let $G$ be a finite group, $X$ a normal subgroup and $Y$ a subgroup such that $Y$ acts faithfully on $X$ by conjugation. Consider the orbit $x^Y$ of an element $x\in X$, then we will call $\widetilde{x^Y}=\sum_{y\in Y} x^y\in \Z X$ the orbit sum of $x$. By $\widetilde{X^Y}$ we will denote the set of all different orbit sums $\widetilde{x^Y}$ for $x\in X$.

\begin{theorem}\label{units pq}
Let $\MC=C_{q^m}\rtimes C_{p^n}$ be a finite metacyclic group with $C_{p^n}=\GEN{b}$ acting faithfully on $C_{q^m}=\GEN{a}$ and with $p$ and $q$ different primes. Assume that either $q\neq 3$, or $n\ne 1$ or $p\ne 2$. For every $j=1,\dots,m$ let $K_j$ and $x_j$ be as above and let $t_j$ be a positive integer such that $t_jx_j^{k}\in \Z \MC$ for all $k$ with $1\leq k\leq p^n$. Then the following two groups are finitely generated nilpotent subgroups of $\mathcal{U}(\Z \MC)$:
$$V_j^+=\GEN{ 1+p^nt_j^2yx_j^{h}\widehat{b}x_j^{-k} \mid y\in \widetilde{\GEN{a}^{\GEN{b}}} ,\ h,k\in\{1,\ldots,p^n\},\ h<k} ,$$
$$V_j^-=\GEN{ 1+p^nt_j^2yx_j^{h}\widehat{b}x_j^{-k} \mid y\in \widetilde{\GEN{a}^{\GEN{b}}} ,\ h,k\in\{1,\ldots,p^n\},\ h>k}.$$
Hence $V^+=\prod_{j=1}^m V_j^+$ and $V^-=\prod_{j=1}^m V_j^-$ are nilpotent subgroups of $\U(\Z \MC)$. Furthermore, the group $$\left \langle U, V^+, V^-\right\rangle,$$ with $U$ as in Theorem~\ref{basisMC}, is of finite index in $\mathcal{U}(\Z \MC)$.
\end{theorem}
\begin{proof}
Since the units of the commutative components are central, we only have to consider the non-commutative components $\Q \MC e$, for $j=1,\dots,m$, of type (\ref{SC2}) with $e=\varepsilon\left(\GEN{a},K_j\right)$. Whereas the center of $\Q \MC e$ coincides with the field of character values of the rational character afforded by the primitive central idempotent $e$ and as this character is induced from a linear character on $\GEN{a}$, it follows that $\Q \MC e \simeq M_{p^n}\left(\Q\left(\widetilde{\GEN{a}^{\GEN{b}}}\right)e\right)$. Let $\O=\Z[ae]^{\GEN{b}}=\Z\left[ \widetilde{\GEN{a}^{\GEN{b}}}\right] e$, which is as a $\Z$-module finitely generated by $\widetilde{\GEN{a}^{\GEN{b}}}e$. Clearly, the elements of the form  $(1-e)+(1+p^nt_j^2yx_j^{h}\widehat{b}x_j^{ k})$ are in $\Z \MC$ and project trivially to $\Q \MC(1-e)$ and by the comments given before the theorem the group $\langle V_j^+,V_j^-\rangle$, generated by these elements, projects to the group $$\langle I+z_jE_{hk}\mid z_j\in p^nt_j^2\O,\ 1\leq h,k\leq p^n,\ i\neq j, E_{hk} \textrm{ a matrix unit} \rangle$$ of elementary matrices of $M_{p^n}(\O)$.

If $p\neq 2$ or $n\neq 1$, then the conditions of Theorem~\ref{elementary} are clearly satisfied. Also if $p=2$, $n=1$ and $q\neq 3$, the conditions are satisfied since $\U(\O)$ is finite if and only if $j=1$ and $q=3$. Hence in all cases $\langle V_j^+,V_j^-\rangle \subseteq \mathcal{U}(\Z \MC)$ is a subgroup of finite index in $\Z(1-e)+\SL_{p^n}(\O)$. By Theorem~\ref{basisMC}, $U$ has finite index in $\mathcal{Z}(\mathcal{U}(\Z \MC))$ and therefore it contains a subgroup of finite index in the center of $\Z(1-e)+\GL_{p^n}(\O)$. Since the center of $\GL_{p^n}(\O)$ together with $\SL_{p^n}(\O)$  generate a subgroup of finite index in $\GL_{p^n}(\O)$, it follows that $\langle U,V^+, V^-\rangle$ contains a subgroup of finite index in the group of units of $\Z(1-e)+\Z \MC e$. Now the statement follows, since $V_j^+$ and $V_j^-$ correspond to upper and lower triangular matrices.
\end{proof}

\begin{remark}
 By well-known results, the hypothesis on $q$ if $p=2$ and $n=1$ can be dropped if we add some more units to the set of generators. Indeed, in this case $\MC=D_{2\cdot 3^m}$ and there is only one ``exceptional" component $M_2(\Q)$. By \cite[Lemma 22.10]{Sehgal1993}, the only additional thing to prove is that the set of generators of $\U(\Z \MC)$ project to a set of generators of $\SL_2(\Z)$. This is satisfied if we add the bicyclic units $1-(1-b)a(1+b),(1-(1-b)a(1+b))^{ba},1+(1-ba)a(1+ba)$ to the generating set (see proof of \cite[Theorem 23.1]{Sehgal1993} for more details).
\end{remark}

\renewcommand{\bibname}{References}
\bibliographystyle{model1-num-names}
\bibliography{references}

\end{document}